\documentclass[twoside, 12pt]{article}
\usepackage{amsthm}
\usepackage{natbib}
\usepackage{color}
\usepackage{graphicx}
\usepackage{amsmath}
\usepackage{amsfonts}
\usepackage{amssymb}
\usepackage{amsxtra}
\usepackage{amstext}
\usepackage{latexsym}
\usepackage{dsfont}
\usepackage{stmaryrd}
\usepackage{mathrsfs}
\usepackage{bbm,bm}
\usepackage{slashbox}

\def \build#1#2#3{\mathrel{\mathop{#1}\limits^{#2}_{#3}}}

\newcommand{\dd}{{\mbox{\boldmath $d$}}}

\newtheorem{theorem}{Theorem}[section]
\newtheorem{remark}{Remark}[section]

\newtheorem{corollary}{Corollary}[section]
\newtheorem{example}{Example}[section]
\newtheorem{lemma}{Lemma}[section]

\def\dd{\mathrm{\,d}}

\def\pp{\mathbb{P}}
\def\rr{\mathbb{R}}
\def\nn{\mathbb{N}}

\title{\bf A model for risk assessment of a large earthquake with application to Chilean data}

\author{Ra\'ul Fierro$^{1,2}$\thanks{Corresponding author: Ra\'ul Fierro, Instituto de Matem\'atica, Universidad de Valpara\'{\i}so, Valpara\'{\i}so, Chile. Emails: raul.fierro@pucv.cl; raul.fierro@uv.cl; rafipra@gmail.com}\,\,\,and V\'ictor Leiva $^3$\\
{\footnotesize $^{1}$Instituto de Matem\'aticas, Universidad de Valpara\'{\i}so, Chile}\\
{\footnotesize $^{2}$Instituto de Matem\'aticas, Pontificia Universidad Cat\'{o}lica de Valpara\'{\i}so, Chile}\\
\author{V\'ictor Leiva}
{\footnotesize $^{3}$Facultad de Ingenier\'ia y Ciencias, Universidad Adolfo Ib\'a\~nez, Chile}}

\date{}

\begin{document} \maketitle

\noindent{\bf Abstract. }  We study the asymptotic distribution for the occurrence time of the next large earthquake, by knowing the last large seismic event occurred a long time ago. We prove that, under reasonable conditions, such a distribution is asymptotically exponential with a rate depending on the asymptotic slope of the cumulative intensity function corresponding to a non-homogeneous Poisson process. Moreover, as it is not possible to obtain an empirical cumulative distribution function for the waiting time of the next large earthquake, a random cumulative function based on existing data is stated.  We demonstrate that analogous results to the theorems of Glivenko-Cantelli and Kolmogorov are satisfied by this random cumulative function. We conduct a simulation study for detecting in what scenario the approximate distribution of the studied elapsed time performs well. Finally, a real-world data analysis is carried out to illustrate the potential applications of our proposal.




\section{Introduction}

\label{Section1} A deep discussion on how possible is to predict an earthquake is given by \cite{Ka97}, but an answer to this question is still open. Throughout the time, some principles has been taken into account to predict the occurrence of earthquakes. Three laws that support these principles are (i) \cite{Om95}'s law, that sets the rate at which aftershocks occur immediately after a large earthquake (main shock); (ii) the elastic rebound law by \cite{Re10}, which indicates that an earthquake must have involved an elastic rebound due to accumulated stress; and (iii) \cite{GR44}'s law, which establishes a relationship between the magnitude and total number of earthquakes.

Time-predictable and slip-predictable models arise from the elastic rebound law and predict the time and slip of the next earthquake, respectively. On the one hand, the time-predictable model assumes that a critical threshold exists, which is constant over time, and once it is attained, there an earthquake should occur. On the other hand, the slip-predictable model assumes that a constant minimum stress is present and all the stress accumulated since the last earthquake, over this minimum, is released in the next seismic event. For details about these models, see \cite{PSSKP11}, \cite{REBKLS12} and \cite{SN80}.  Other authors, such as \cite{REBKLS12}, argue that  a ``memoryless'' earthquake model with fixed inter-event time or fixed slip is better than time-and-slip-predictable models for earthquake occurrence. Based on the slip-predictable model, the next earthquake should have a high magnitude, if the last seismic event occurred a long time ago. For this reason, it is important to know, as accuracy as possible, the distribution of the occurrence time of the next large earthquake, mainly in the case that the last seismic event occurred a long time ago. This is the main motivation for studying the distribution of the waiting time for the next large earthquake, by knowing that a long time has elapsed since the last seismic event.

The main objective of this work is to establish the limit distribution of the occurrence time of the next large earthquake, for a specific geological zone and given that a long time has elapsed from the last large earthquake. We prove that, under reasonable conditions, this distribution is exponential with rate depending on the asymptotic slope of the cumulative intensity function (CIF) of a non-homogeneous Poisson process (NHPP). Our assumptions are quite simple and basically consist of supposing that between the last large earthquake and the next one, $k-1$ ($k\geq2$) seismic events of minor intensity occur. Of course, it is not possible to know this parameter $k$ and any characteristic, of the distribution for the waiting time for the next large earthquake, only should be based on the information provided by the events of minor intensity. We prove that the limit distribution of the waiting time for the next large earthquake does not depend on $k$, whenever a long time is elapsed from the last large earthquake, and it is exponential with a parameter depending on the asymptotic slope of the CIF. An estimator for this parameter is provided using existing data from the last earthquake. Since it is not possible to obtain an empirical cumulative distribution function (CDF) for the waiting time of the next large earthquake, this estimator is used to state a random CDF, which is proved that satisfies analogous results to the known theorems of Glivenko-Cantelli and Kolmogorov.

A number of statistical tools have been used in studying seismic activity; see, for example, \cite{kklc:14}, \cite{fsi:15} and \cite{kr:15}. Other works on this matter based on point processes are attributed to \cite{Og88} and \cite{ac:15}. Indeed, \cite{Og88} modeled seismic activity by means of a Hawkes process and, recently, \cite{Fi15} and \cite{FLM15} introduced variants of the Hawkes process, which could be more appropriate for modelling earthquakes. Also in \cite{FLRS13}, the asymptotic distribution of a shock model based on a nonhomogeneous Poisson process could be adapted for this purpose. However,
we appreciate  the methodology that we are introducing in this paper due to its simplicity.

The paper is organized as follows. In Section \ref{section2}, we provide some notations and facts on which the results in later sections rely. In Section \ref{section3}, we state conditions for the existence of the limit CDF. In Section \ref{section4}, we study an estimator for the asymptotic slope of the CIF associated with the NHPP and its asymptotic properties. In Section \ref{section5}, we prove that this estimator is of maximum likelihood (ML), and based on it, a random CDF is defined. Moreover, we state and demonstrate some properties of this estimator for the limit CDF. In Section \ref{section6}, we carry out a simulation study for detecting which is the suitable scenario for the proposed approximate distribution performs well. In Section \ref{section7}, we conduct a real-world data analysis to illustrate the potential applications of our proposal. Finally, in Section \ref{section8}, we provide some conclusions about this study.

\section{Preliminaries}\label{section2}

According to \citet[][p.\,700]{WSSM84}: ``A correct representation of seismic hazard due to a fault must take into account the time elapsed since the most recent rupture''. From this quotation, the main concern for assessing risk of large earthquakes should be focused on determining the probability that the rupture time $T$ of a fault should occur during the next $h$ years, conditional to $t$ years have elapsed since the last rupture. This conditional probability is given by
\begin{equation}\label{e1}
F_t(h) = \pp(t<T\leq t+h|T>t), \quad t > 0, h > 0.
\end{equation}
A number of authors, such as \cite{DN86}, \cite{Di88} and \cite{YTRR07}, have taken into account the probability expressed in \eqref{e1} for assessing seismic risk. We are also interested in this assessment, which from our point of view involves studying the probability that a large earthquake occurs whether a long time has elapsed since the last large earthquake. Consequently, when a long time has elapsed, a natural criterion for this assessment consists of assuming the time $S$ of the next large earthquake is a random variable with CDF $G$ satisfying
\begin{equation}\label{e2}
G(h)=\lim_{t\to\infty} F_t(h), \quad h > 0,
\end{equation}
where, of course, we have to assume that $G$ is a CDF.

A realistic model considers only distributions for $T$ such that $G$ corresponds to a non-degenerate distribution. This assumption seems to be quite reasonable, because if a long time has elapsed since the last large earthquake, one should not expect that the next event occurs right now. By assuming this non-degeneracy, we conclude that a number of distributions considered by some authors, such as Gumbel \citep{Co68,KK77}, lognormal \citep{NB87} and Weibull \citep{HM13,Ri76} distributions, it does not seem to be the appropriate distribution for times between seismic events, because each of them lead to degenerate distributions for $G$ given by \eqref{e2}. Otherwise, our assumption is not correct or, at least, it is contradicted with the use of these distributions.

A possible distribution for $T$ which satisfies the above requirements can be obtained as follows. Let $\{T_n\}_{n\in\nn}$ be an increasing sequence of stoping times corresponding to an NHPP with intensity function (IF) $\lambda:\rr_+\to\rr_+$ and a CIF $\Lambda:\rr_+\to\rr_+$ given by
\begin{equation}\label{e3}
\Lambda(t)=\int_0^t\lambda(u)\dd u.
\end{equation}
We propose a criterion based on the assumption that some earthquakes of medium and large intensity, in a specific geological zone, occur according to an NHPP with a CIF as given in \eqref{e3}, which can be estimated of different ways, as we see later. We are assuming the occurrence time of a large earthquake is given by a sum of random variables, which represents the times between the occurrence of consecutive earthquakes. By supposing the large earthquake exactly occurs at the $k$th shock, this random time is denoting by $T_k$, the $k$th jump time of the NHPP. However, an important difficulty could arise due to it is not possible to know the value of $k$, as mentioned. As we will see, our assumptions allow us to obtain a non-degenerate CDF as defined in \eqref{e2} for $S$, which does not depend on $k$.

The time $T_k$ turns out to have a kind of gamma distribution with parameters depending on $k$ and the  CIF $\Lambda$. Indeed, we obtain the probability density function (PDF) of $T_k$ as
\begin{equation}\label{e4}
    f_k(t)=\left\{
\begin{array}{ccl}
  \frac{\lambda(t)}{(k-1)!}\Lambda(t)^{k-1}\exp(-\Lambda(t))& \text{if} & t\geq0;\\
  0& \text{if} & t<0;
\end{array} \right.
\end{equation}
where $\lambda$ is the IF of the CIF defined by \eqref{e3}. In the sequel, we assume that $\lambda$ is continuous so that the derivative of $\Lambda$ coincides with $\lambda$. Moreover, we refer to the CDF $G$, defined in \eqref{e2}, as the limit CDF of $T$. In the next section, we calculate the limit CDF of $T_k$.

\section{On the limit cumulative distribution function}\label{section3}

When $T$ follows an exponential distribution, condition \eqref{e2} is trivially satisfied. However, it seems that it is not appropriate to assume a priori that the distribution of $T$ is exponential. The following theorem establishes that, under a mild condition, when the limit CDF of $T$ exists, this corresponds to the exponential distribution or to the degenerate distribution at zero.

\begin{theorem}\label{t1}
Suppose $T$ has a continuous PDF $f$, a limit CDF  $G$, which is continuously differentiable at zero from the right, and its derivative $G'(0) = m$. Then, for each $h\geq0$,
\begin{itemize}
  \item [(i)] $G(h)=1 - \lim_{t\to\infty}\frac{f(t+h)}{f(t)}$ and
  \item [(ii)] $ G(h)=1-\exp(-m h)$.
\end{itemize}
\end{theorem}

\begin{proof} Let  $F$ be the CDF of $T$. Then, by the l'H\^{o}pital rule and the fundamental theorem of calculus, we have
\begin{eqnarray*}
G(h) & = & \lim_{t\to\infty} \pp(t<T\leq t+h|T>t) \\
     & = & 1 - \lim_{t\to\infty}\frac{1-F(t+h)}{1-F(t)}\\
     & = & 1 - \lim_{t\to\infty}\frac{F'(t+h)}{F'(t)} \\
     & = & 1 - \lim_{t\to\infty}\frac{f(t+h)}{f(t)}.\\
\end{eqnarray*}
Consequently,
\begin{eqnarray*}
    \frac{G(h+\Delta h)-G(h)}{\Delta h}  & = & \frac{1}{\Delta h}\left(\lim_{t\to\infty}\frac{(f(t+h)-f(t+h+\Delta
    h))}{f(t+h)} \frac{f(t+h)}{f(t)}\right) \\
      & = & \frac{1}{\Delta h}\left(1-\lim_{t\to\infty}\frac{f(t+h+\Delta h)}{f(t+h)}\right)\lim_{t\to\infty}\frac{f(t+h)}{f(t)}.\\
\end{eqnarray*}
Thus,
$$
\frac{G(h+\Delta h)-G(h)}{\Delta h}=\frac{G(\Delta h)(1-G(h))}{\Delta h}.
$$
and by taking limit as $\Delta h \to 0$,  we have $G'(h)=G'(0)(1-G(h))$. This differential equation has a unique solution, for $G'(0)=m$, which is given by $G(h)=1-\exp(-m h)$. Therefore, the proof is complete.
\end{proof}

\begin{remark}\label{r1}
Let $f$ be as in Theorem \ref{t1} and suppose that there exists $t_0>0$ such that, for each $t\geq t_0$, $f(t)>0$ and $f$ is decreasing on $[t_0,\infty)$. Hence, from (i) in Theorem \ref{t1}, for each $t\geq t_0$, the function $G_t:\rr_+\to\rr_+$ defined as $G_t(h)=1- f(t+h)/f(t)$ is a CDF. The family $\{G_t\}_{t>0}$ of CDFs is considered later to conduct some simulations.
\end{remark}

\begin{remark}\label{r2}
As in proof of Theorem \ref{t1}, for each $h>0$, $G'(h)=G'(0)(1-G(h))$, we have $G'(0)>0$ whenever $G$ is not constant. Hence, in this case, $G$ corresponds to an exponential CDF with parameter $\lambda=G'(0)$. Otherwise, $G$ corresponds to the degenerate  distribution at zero.
\end{remark}

\begin{example}\label{ex1} Let $T$ be a random variable having gamma distribution with parameters $\lambda >0$ and $\beta>0$, that is, the PDF of $T$ is given by
$$
f(t)=\frac{\lambda(\lambda t)^{\beta-1}\exp({-\lambda t})}{\Gamma(\beta)}, \quad t>0,
$$
where $\Gamma$ is the usual gamma function. Then, $T$ has a limit CDF given by
$$
G(h)=1-\exp(-\lambda h), \quad h\geq0.
$$
\end{example}

\vspace{0.25cm}
The next theorem provides the limit CDF of $T_k$, corresponding to the occurrence time of the $k$th large earthquake, such as defined in Section \ref{section2}.

\begin{theorem}\label{t2}
Suppose $T$ has a PDF given by \eqref{e4} and there exists $\lim_{t\to\infty}\lambda(t)=m$, where $m$ is a strictly
positive real constant. Then, the limit CDF of $T_k$ is given by $$
G(h)=1-\exp(-mh), \quad h\geq0,
$$
with $m$ not depending on $k$.
\end{theorem}

\begin{proof}
From l'H\^{o}pital rule, we have
$$
\lim_{t\to\infty}\frac{\Lambda(t+h)}{\Lambda(t)}=\lim_{t\to\infty}\frac{\lambda(t+h)}{\lambda(t)}=1.
$$
In addition, from the mean value theorem, there exists $\xi(t)$ between $t$ and $t+h$, with $t >0$ and $h\geq0$, such that
$$ \Lambda(t+h)-\Lambda(t)=\lambda(\xi(t))h.
$$
Consequently,
$$ \frac{f_k(t+h)}{f_k(t)}=\exp(-\lambda(\xi(t))h)\left(\frac{\Lambda(t+h)}{\Lambda(t)}\right)^{k-1}\frac{\lambda(t+h)}{\lambda(t)}.
$$
Thus, as
$$
\lim_{t\to\infty}\frac{f_k(t+h)}{f_k(t)}=\lim_{t\to\infty}\exp(-\lambda(\xi(t))h)=\exp(-m h),
$$
the proof follows from (i) in Theorem \ref{t1}.
\end{proof}

\begin{remark}\label{r3}
We call the constant $m$ in Theorem \ref{t2} as the asymptotic slope of $\Lambda$.
\end{remark}

\section{Asymptotic slope of the cumulative intensity}\label{section4}

Let $N$ denote the NHPP with CIF $\Lambda$ defined in \eqref{e3}. We are assuming the PDF of the occurrence time of the next large earthquake is given by \eqref{e4}. Then, it follows from Theorem \ref{t2} that the asymptotic slope $m$ of $\Lambda$, when it exists,  is the unique parameter of the limit CDF. This parameter  needs to be estimated and due to $\Lambda$ can be estimated by means of $N$, the following theorem is a first approach in this direction.

\begin{theorem}\label{t3}
Suppose  $\lim_{t\to\infty}\lambda(t)=m$. Then, $\lim_{t\to\infty}\Lambda(t)/t=m$.
\end{theorem}

\begin{proof}
For each $\varepsilon>0$, let $t_0>0$, such that $|\lambda(s)-m|<\varepsilon$, whenever $s\geq t_0$. Consequently, for each $t>t_0$, we have $$
\begin{array}{ccl}
  \left|\frac{1}{t}\int_0^t\lambda(s)\dd s -m\right| & = & \left|\frac{1}{t}\int_0^t(\lambda(s)-m)\dd s\right|  \\
    & \leq &\frac{1}{t}\int_0^t|\lambda(s)-m|\dd s  \\
    & \leq &\frac{1}{t}\int_0^{t_0}|\lambda(s)-m|\dd s + \frac{1}{t}\int_{t_0}^t|\lambda(s)-m|\dd s\\
    & < &\frac{1}{t}\int_0^{t_0}|\lambda(s)-m|\dd s +\varepsilon.\\
\end{array} $$ By taking limit as $t\to\infty$, we obtain $$ \limsup_{t\to\infty}\left|\frac{1}{t}\int_0^t\lambda(s)\dd s -m\right|\leq\varepsilon. $$
As $\varepsilon>0$ is arbitrary, we have
$$
\lim_{t\to\infty}\frac{\Lambda(t)}{t} = m,
$$
which concludes the proof.
\end{proof}

We are just assuming existence of the asymptotic slope of the CIF to provide more generality to the distribution of $T_k$. However, the results of this work maintain their importance whether we assume that the IF of the Poisson process is constant and equal to $m$. In this case, $T_k$ has Erlang distribution and its PDF is given by \eqref{e4} by replacing $\lambda(t)$ by $m$. Hence, in practical terms, instead of knowing the asymptotic slope of the CIF, it suffices to know or estimate a constant value of $\lambda(t)$, for all $t\geq \tau^*$, where $\tau^*\geq0$ is a large enough time instant. The following corollary aims to this end.

\begin{corollary}\label{c1}
Let $\tau^*\geq0$, $m>0$ and suppose that, for each $t\geq\tau^*$, $\lambda(t)=m$. Then,
$$
\lim_{t\to\infty}\frac{\Lambda(t)-\Lambda(\tau^*)}{t-\tau^*} = m.
$$
\end{corollary}

From Corollary \ref{c1}, an estimate of the CIF allow us to estimate its asymptotic slope. Estimates of the CIF for an NHPP have been investigated by a number of authors, such as \cite{AL00}, \cite{He03}, \cite{Le91,Le04} and \cite{LS79}, where some of these estimates are given by the own NHPP. Hence, the following theorem is quite useful to this purpose.

\begin{theorem}\label{t4}
Let $\tau^*\geq0$ and suppose that $\lim_{t\to\infty}\lambda(t)=m$ exists. Then, the following two conditions hold:
\begin{itemize}
\item [(i)] $\lim_{t\to\infty}{(N_t-N_{\tau^*}})/{(t-\tau^*)}=m$, $\pp$-almost surely (a.s.) and

\item [(ii)] $\sqrt{t-\tau^*}\left(\frac{N_t-N_{\tau^*}}{t-\tau^*}-m\right)\mathop{\to}\limits^{\mathcal{D}}\,{\rm N}(0,m)$ as $t \to\infty$, whenever $\lim_{t\to\infty}\sqrt{t}\left(\frac{\Lambda(t)}{t}-m\right)=0$, where $\mathop{\to}\limits^{\mathcal{D}}$ stands for convergence in distribution and ${\rm N}(0,m)$ is a normal random variable with mean zero and variance $m$.
\end{itemize}
\end{theorem}

\begin{proof}
Let $M=\{M_t;t\geq0\}$ be the martingale defined by $M_t=N_t-\Lambda(t)$. As
$$
\frac{N_t-N_{\tau^*}}{t-\tau^*}=\frac{\Lambda(t)-\Lambda(\tau^*)}{t-\tau^*}\left(\frac{M_t-M_{\tau^*}}{\Lambda(t)-\Lambda(\tau^*)}+1\right),
$$
and Theorem 8.2.17 in \cite{DCD86} implies $\lim_{t\to\infty}(M_t-M_{\tau^*})/(\Lambda(t)-\Lambda(\tau^*))=0$, $\pp$-a.s., from Theorem \ref{t3}, we have $\lim_{n\to\infty}{(N_t-N_{\tau^*})}/{(t-\tau^*)}=m$, $\pp$-a.s.

Suppose $\lim_{t\to\infty}\sqrt{t}({\Lambda(t)}/{t}-m)=0$ and let $L$ be a homogeneous Poisson process (HPP) with rate equal to one.  For each $t\geq\tau^*$, we have
$$
 \hspace{-3.8in}\sqrt{t-\tau^*}\left(\frac{L_{\Lambda(t)}-L_{\Lambda(\tau^*) }}{t-\tau^*}-m\right)  =
$$
$$\hspace{2in}\left(\frac{\Lambda(t)}{t-\tau^*}\right)^{1/2} \left(\frac{L_{\Lambda(t)}-\Lambda(t)}{\sqrt{\Lambda(t)}}\right)+\sqrt{t-\tau^*}\left(\frac{\Lambda(t)-L_{\Lambda(\tau^*) }}{t-\tau^*}-m\right).
$$
Consequently, in order to prove that
\begin{equation}\label{e5}
  \sqrt{t-\tau^*}\left(\frac{L_{\Lambda(t)}-L_{\Lambda(\tau^*) }}{t-\tau^*}-m\right)\mathop{\to} \limits^{\mathcal{D}}\mathrm{N}(0,m), \quad \mbox{as} \quad t \to\infty,
\end{equation}
and since $\lim_{t\to\infty} L_{\Lambda(\tau^*)}/\sqrt{t-\tau^*}=0$, we need to prove
\begin{equation}\label{e6}
  \left(\frac{\Lambda(t)}{t-\tau^*}\right)^{1/2} \left(\frac{L_{\Lambda(t)}-\Lambda(t)}{\sqrt{\Lambda(t)}}\right)\mathop{\to} \limits^{\mathcal{D}}\mathrm{N}(0,m),\quad \mbox{as}\quad t \to\infty,
\end{equation}
and
\begin{equation}\label{e7}
  \lim_{t\to\infty}\sqrt{t-\tau^*}\left(\frac{\Lambda(t)}{t-\tau^*}-m\right)=0.
\end{equation}
As $(L_t-t)/\sqrt{t}\mathop{\to} \limits^{\mathcal{D}}\mathrm{N}(0,1)$ and $\lim_{t\to\infty}\Lambda(t)/(t-\tau^*)=m$, condition \eqref{e6} follows. From the intermediate value theorem, for each $t>0$, there exists $\xi(t)$ between $t-\tau^*$ and $t$ such that $\Lambda(t) = \Lambda(t-\tau^*)+\lambda(\xi(t))\tau^*$. This fact implies that
$$
\sqrt{t-\tau^*}\left(\frac{\Lambda(t)}{t-\tau^*}-m\right)=\sqrt{t-\tau^*}\left(\frac{\Lambda(t-\tau^*)}{t-\tau^*}-m\right)+
\frac{\lambda(\xi(t))\tau^*}{\sqrt{t-\tau^*}}
$$
and since we are assuming $\lim_{t\to\infty}\sqrt{t-\tau^*}({\Lambda(t-\tau^*)}/{(t-\tau^*)}-m)=0$ and $\lambda$ is bounded, condition \eqref{e7} holds.  Hence, we have proven condition \eqref{e5}, but due to $N_t-N_{\tau^*}$ and $L_{\Lambda(t)}-L_{\Lambda(\tau^*)}$ have the same distribution, for each $t>0$, we have
$$
\sqrt{t-\tau^*}\left(\frac{N_t-N_{\tau^*}}{t-\tau^*}-m\right)\mathop{\to}\limits^{\mathcal{D}}\,{\rm N}(0,m),\quad \mbox{as}\quad t \to\infty,
$$
which completes the proof.
\end{proof}

\begin{corollary}\label{c2}
Let $\tau^*\geq0$ and suppose that $\lambda(t)=m$, for each $t\geq\tau^*$. Then, the family $\{\widehat{m}_{t}\}_{t>\tau^*}$, defined as $\widehat{m}_{t}=(N_t-N_{\tau^*})/(t-\tau^*)$, is asymptotically normal distributed with mean $m$ and variance $m/(t-\tau^*)$.
\end{corollary}

\section{ML estimation and a type of empirical CDF}\label{section5}

Let $\tau^*>0$ and assume that, for each $t\geq\tau^*$, $\lambda(t)=m$. Then, the CIF of the NHPP has the form
$$
\Lambda(t)= \left\{
\begin{array}{ccl}
  \int_0^t\lambda(s)\dd s, & \mbox{if} & 0\leq t<\tau^*; \\
  \int_0^t\lambda(s)\dd s+(t-\tau^*)m, & \mbox{if} &  \tau^*\leq t.
\end{array}
\right.
$$
In this section, the parameter $m$ is estimated by means of the ML method. Fix $t>\tau^*$ and let $P_t$ be the distribution of $N$ on the Skorohod space $\mathrm{D}([0,t],\mathbb{R})$ of the right continuous functions from $[0,t]$ to $\mathbb{R}$, which have left limits. Hence, $P_t$ is absolutely continuous with respect to $Q$, the distribution of an HPP with rate equal to one. From Theorem 3 in \cite{Br81} or Theorem 2.31 in \cite{Ka91}, the Radon-Nikodym derivative of $P_t$ with respect to $Q$, for $t>\tau^*$, evaluated at $N$, is given by
\begin{equation} \label{e8}
\footnotesize{
\frac{\dd P_t}{\dd Q}(N)=\exp\left(\int_0^{\tau^*}\log(\lambda(u))\dd N_u+\log(m)(N_{t}-N_{\tau^*})-\int_0^{\tau^*}(\lambda(u)-1)\dd u-(t-\tau^*)(m-1)\right).}
\end{equation}
It is easy to see from \eqref{e8} that the ML estimator of $m$, for $\tau>\tau^*$,  is given by
$$
\widehat{m}_{\tau}=\frac{N_{\tau}-N_{\tau^*}}{\tau-\tau^*} ,
$$
which coincides with the estimator of $m$ given in Corollary \ref{c2}. Consequently, the limit distribution $G$, on the time interval $[0,\tau]$, can be estimated by means of a random CDF $\widehat{G}_\tau$ defined for $h\geq0$ as $\widehat{G}_\tau(h)=1-\exp(-\widehat{m}_\tau h)$.
Even thought $\widehat{G}_\tau$ is not properly the typical empirical CDF based on independent random variables, a
Glivenko-Cantelli type theorem can be established as follows.

\begin{theorem}\label{p1} Under assumptions and notations stated in this section, we have
$$
\lim_{\tau\to\infty}\sup_{h\geq0}|\widehat{G}_\tau(h)-G(h)|=0,\quad \pp\mbox{-a.s.}
$$
\end{theorem}

\begin{proof}
From (i) in Theorem \ref{t4}, for each $h>0$, $\widehat{G}_\tau(h)\build{\longrightarrow}{}{\tau\to\infty} G(h)$, $\pp$-a.s. Hence, it follows from the P\'olya Lemma \citep{r:97} that, as $\tau \to \infty$, $\{\widehat{G}_\tau; \tau>0\}$ converges, $\pp$-a.s., uniformly to $G$, which concludes the proof.
\end{proof}

Also, a Kolmogorov type theorem is given  below.

\begin{theorem}\label{p2} Let $\tau>0$. Under assumptions and notations stated in this section, we have
$$
\sup_{h\geq0}|\sqrt{\tau}(\widehat{G}_\tau(h)-G(h))|\build{\longrightarrow}{\mathcal{D}}{\tau\to\infty}\left|{\rm N}\left(0,\frac{\exp(-2)}{m}\right)\right|,
$$
where ${\rm N}(0,\exp(-2)/m)$ is a normal random variable with mean zero and variance $\exp(-2)/m$.
\end{theorem}

\begin{proof}
From the intermediate value theorem, there exists $\xi_\tau$ between $\widehat{m}_\tau$ and $m$ such that
$$
\widehat{G}_\tau(h)-G(h)=h\exp(-\xi_\tau h)(m-\widehat{m}_\tau).
$$
Then,
$$
\sup_{h\geq0}|\sqrt{\tau}(\widehat{G}_\tau(h)-G(h))|=\frac{\exp(-1)}{\xi_\tau}|\sqrt{\tau}(\widehat{m}_\tau-m)|.
$$
However, from (i) and (ii) of Theorem \ref{t4}, $\lim_{\tau\to\infty}\xi_\tau=m$, $\pp$-a.s., and $\sqrt{\tau}(\widehat{m}_\tau-m)\mathop{\to}\limits^{\mathcal{D}}\,{\rm N}(0,m)$, respectively. Therefore, the proof follows from the Slutsky theorem.
\end{proof}

\section{{Simulation}}\label{section6}

Let $\{G_t\}_{t>0}$ be the family of CDFs defined in Remark \ref{r1} and $G$ be the corresponding limit CDF with $m=G'(0)>0$. In order to evaluate the  accurate of the asymptotic approximation of $G_t$  by means of $G$, we simulate $n$ data from a random variable $S$ with CDF $G_t$, for different values of $t\geq0$. By assuming the distribution of $T_k$ is given by \eqref{e3} with $\Lambda(t)=mt$ and $m>0$, we have $G_t$ is expressed as
\begin{equation}\label{e9}
  G_t(h)=1-\left(1+\frac{h}{t}\right)^{k-1}\exp(-mh).
\end{equation}
It is easy to note that, for each $t\geq k-1$, $G_t$ given in \eqref{e9} is a CDF. Our simulation study consists of (i) partitioning the positive part of the real straight line in $r$ subintervals, which are determined as $0=h_0<h_1<\cdots<h_{r-1}< h_r=\infty$, (ii) determining the observed percentage of times that the $n$ simulated values of $S$ fall into each subinterval $[h_{i-1},h_i[$ and (iii) computing the corresponding expected percentage given by $100\times(G(h_i)-G(h_{i-1}))$, for $i=1,\dots,r-1$, and by $100\times(1-G(h_{r-1}))$, for $i=r$. From the probability integral transform, $G(S)$ follows a uniform distribution in the interval $[0, 1]$. Hence, from (i) in Theorem \ref{t1}, we expect, for large enough values of $t$, $G_t(S)$ to have approximately a uniform distribution. The simulation is carried out with $n=1000$ and $r = 10$ class intervals, because it is coherent with the Sturges rule; see \cite{St26}. This rule indicates that a suitable number of class intervals is $r =  1 + \log_2(n)$, which in our case is $r =  1 + \log_2(1000) \approx 10$. We expect that, as $t$ increases, percentages of simulated and expected  values must be similar. A goodness-of-fit $\chi^2$ test is used to evaluate this similarity. Indeed, based on such a test, our objective is to determinate values of $t$ for which the approximation provided in Theorem \ref{t2} is satisfactory. Specifically, the scenario of the simulation study considers $m = 1$ and $k= 10$.

The random variable $S$ is simulated $1000$ times, with $r = 10$, where $h_1,\dots,h_9$ are chosen in such a way that $G(h_i)=i/10$, that is,
$
h_1=0.1053605, h_2=0.2231436, h_3=0.3566749, h_4=0.5108256, h_5=0.6931472, h_6=0.9162907, h_7=1.2039728, h_8=1.6094379$ and $h_9=2.3025851$.
Then, the observed percentages falling into these subintervals are determined; see Table \ref{tab:1}. The expected percentages are all $10\%$.

\begin{table}[h!]\label{tab:1}
\centering
\renewcommand{\arraystretch}{1}
\renewcommand{\tabcolsep}{0.09cm}
\caption{
percentages of simulated values of $S$ for the indicated time $t$ and p-values of the corresponding $\chi^2$ test ($m = 1$ and $k=10$).}
\begin{tabular}{c ccccccccccc}
\hline
&\multicolumn{11}{c}{$h$}\\
$t$&$[h_0,h_1[$&$[h_1,h_2[$&$[h_2,h_3[$&$[h_3,h_4[$&$[h_4,h_5[$&
$[h_5,h_6[$&$[h_6,h_7[$&$[h_7,h_8[$&$[h_8,h_9[$&$[h_9,h_{10}]$&p-value\\
\hline
10 & 1.1 &  1.7  &  1.2  &  2.0  &  1.4  &  3.4  &  4.9   &  7.4  &  12.0 &  64.9  &$<$ 0.001\\
20 &5.0  &  6.3  &  6.8  &  6.4  &  6.8  &  9.2  &  10.6  &  10.1 &  13.4 &  25.4  &$<$ 0.001\\
25 &6.7  &  6.4  &  7.0  &  8.3  &  7.4  &  9.2  &  10.1  &  11.8 &  12.5 &  20.6  &0.057\\
30 &7.5  &  9.0  &  8.4  &  7.4  &  7.7  &  7.5  &  9.9   &  10.1 &  13.0 &  19.5  &0.175\\
40 &7.3  &  8.9  &  9.4  &  9.2  &  9.6  &  8.7  &  11.7  &  8.2  &  11.1 &  15.9 &0.803\\
50 &10.3 &  9.7  &  8.2  &  8.8  &  9.3  & 11.5  &  8.9   &  9.7  &  11.1 &  12.5 &0.996\\
\hline
\end{tabular}
\end{table}

\newpage
As mentioned,  closeness of $10\%$ is evaluated by using the $\chi^2$ test. For this purpose, we define the statistic
$$
\chi^2 = \sum_{j=1}^{10}(O_j-10)^2 \sim \chi^2(9),
$$
where, for $j = 1,\dots,10$, $O_j$ is the observed value of $S$ falling in the $j$th interval. For each $i\in\{1,\dots,6\}$, we define
$$
\chi^2_i = \sum_{j=1}^{10}(A_{ij}-10)^2,
$$
where
$$
(A_{ij})=
\left(
  \begin{array}{cccccccccc}
 1.1  &  1.7  &  1.2  &  2.  &  1.4  &  3.4  &  4.9  &  7.4  &  12  &  64.9 \\
5.  &  6.3  &  6.8  &  6.4  &  6.8  &  9.2  &  10.6  &  10.1 &   13.4 &  25.4  \\
6.7  &  6.4  &  7.  &  8.3  &  7.4  &  9.2  &  10.1  &  11.8  &  12.5  &  20.6\\
 7.5  &  9.  &  8.4  &  7.4  &  7.7  &  7.5  &  9.9  &  10.1  &  13.  &  19.5 \\
7.3  &  8.9  &  9.4  &  9.2  &  9.6  &  8.7  &  11.7  &  8.2  &  11.1  &  15.9\\
10.3  &  9.7  &  8.2  &  8.8  &  9.3  &  11.5  &  8.9  &  9.7  &  11.1  &  12.5\\
  \end{array}
\right),
$$
for $1\leq i\leq 6$ and $1\leq j\leq 10$. The p-value corresponding to the $i$th row of Table \ref{tab:1} is obtained as $\pp(\chi^2\geq\chi^2_i)$. Thus, the $\chi^2$ test allows us to evaluate the goodness-of-fit of $G_t$ by means of the limit CDF $G$; see Table \ref{tab:1}. From this table, we conclude that, for $t\geq25$, the distribution of $G_t$ is well approximated by $G$.

\section{Data analysis}\label{section7}

Chile is a country with a high seismic activity and enough data have been registered about this activity. For this reason, we choose this country to apply our results by means of a data analysis. Specifically, we study the seismic activity in the north zone of Chile, which  is known by some seismologists as Area A. In this zone, 39 earthquakes were registered between the years 1604 and 2007, whose magnitudes fluctuate between 7.0 and 8.9 Richter degrees ($^\circ$R). All of these earthquakes with their respective dates of occurrence are described in Table \ref{tab:data}.

\textcolor{red}{
\begin{table}[h!]
\centering
\footnotesize
\renewcommand{\tabcolsep}{0.03cm}
\caption{year and Richter degree for data of earthquakes in the Area A of the north zone of Chile.}\label{tab:data}
{
\begin{tabular}{ccccccccccccccccccccccc}
\hline
Year&1604&1615&1681&1715&1768&1831&1833&1836&1868&1870&1871&1876&1877&1878&1905&1906&1906&1909&1911&1925\\
$^\circ$R&8.5*&7.5&7.4&8.8*&7.7&7.6&7.8&7.5&8.0&7.5&7.5&7.2&8.6*&7.3&7.0&7.2&7.0&7.6&7.3&7.3\\
Year&1928&1933&1936&1940&1945&1947&1948&1953&1956&1965&1966&1967&1970&1983&1987&1988&1995&2005&2007&$\to$\\
$^\circ$R&7.1&7.6&7.3&7.3&7.2&7.0&7.0&7.5&7.1&7.1&7.9&7.5&7.0&7.4&7.2&7.0&7.6&7.9&7.6&\\
\hline
\end{tabular}
}
\end{table}}

\newpage
According to considerations in \cite{GR44}, large earthquakes can be classified into two groups. The first of them (G1) corresponds to destructive large earthquakes of medium intensity fluctuating between 7.0 $^\circ$R and 8.4 C. The second of these groups (G2) corresponds to the largest recorded earthquakes with intensity within a range starting at 8.5 $^\circ$R and having no upper limit. We detect 36 and 3 large earthquakes belong to G1 and G2, respectively. The times when any of the three large earthquakes from G2 occur are considered as zero. The times of occurrence of large earthquakes from G1 until a large earthquake from G2 occurs are registered in Tables \ref{tab:3}, \ref{tab:4} and \ref{tab:5}, which we call first, second and third data set, respectively. The first data set is not be considered in this analysis due to it contains too few information; see Table \ref{tab:3}. The corresponding estimated asymptotic slope of the CIF $\widehat{m}_t$ for each of these data sets is also given in Tables \ref{tab:3}, \ref{tab:4} and \ref{tab:5}. For $t=53,116,118,121$ provided in Table \ref{tab:4}, the empirical CDF $\widehat{F}_t$ associated with $F_t$ is given in Table \ref{tab:6}. In this table, we compare numerically these empirical CDFs with the corresponding random CDF.

\begin{table}[h!]
\centering
\small
\renewcommand{\arraystretch}{1.25}
\caption{first data set.}\label{tab:3}
\begin{tabular}{cccc}\hline
$t$&0&11&77\\\hline
$\widehat{m}_t$&0&$\frac{1}{11}$&$\frac{2}{77}$\\\hline
$^\circ$R&8.5*&7.5&7.4\\\hline
\end{tabular}
\end{table}

\begin{table}[h!]
\centering
\small
\renewcommand{\arraystretch}{1.25}
\caption{second data set.}\label{tab:4}
\vspace{0.1cm}
\begin{tabular}{cccccccccc}\hline
$t$&0&53&116&118&121&153&155&156&161\\\hline
$\widehat{m}_t$&0&$\frac{1}{53}$&$\frac{2}{116}$&$\frac{3}{118}$&$\frac{4}{121}$&$\frac{5}{153}$&$\frac{6}{155}$&$\frac{7}{156}$&$\frac{8}{161}$\\\hline
$^\circ$R&8.8*&7.7&7.6&7.8&7.5&8.0&7.5&7.5&7.2\\\hline
\end{tabular}
\end{table}

\begin{table}[h!]
\centering
\small
\renewcommand{\arraystretch}{1.25}
\renewcommand{\tabcolsep}{0.03cm}
\caption{third data set.}\label{tab:5}
\vspace{0.1cm}
{
\begin{tabular}{ccccccccccccccccccccccccccccc}\hline
$t$&0&1&28&29&29&32&34&48&51&56&59&63&68&70&71&76&79&88&89&90&93&106&110&111&118&128&130&$\to$\\\hline
$\widehat{m}_t$&0&1&$\frac{2}{28}$&$\frac{3}{29}$&$\frac{4}{29}$&$\frac{5}{32}$&$\frac{6}{34}$&$\frac{7}{48}$&
$\frac{8}{51}$&$\frac{9}{56}$&$\frac{10}{59}$&$\frac{11}{63}$&$\frac{12}{68}$&$\frac{13}{70}$&
$\frac{14}{71}$&$\frac{15}{76}$&$\frac{16}{79}$&$\frac{17}{88}$&$\frac{18}{89}$&$\frac{19}{90}$&$\frac{20}{93}$&
$\frac{21}{106}$&$\frac{22}{110}$&$\frac{23}{111}$&$\frac{24}{118}$&$\frac{25}{128}$&$\frac{26}{130}$&\\\hline
$^\circ$R&8.6*&7.3&7.0&7.2&7.0&7.6&7.3&7.3&7.1&7.6&7.3&7.3&7.2&7.0&7.0&7.5&7.1&7.1&7.9&7.5&7.0&7.4&7.2&7.0&7.6&7.9&7.6&\\\hline
\end{tabular}
}
\end{table}

\begin{table}[h!]
\centering
\footnotesize
\renewcommand{\arraystretch}{0.8}
\renewcommand{\tabcolsep}{0.2cm}
\caption{\small values of empirical and estimated CDFs and their differences for the indicated time.}\label{tab:6}
\begin{tabular}{cccccccccc}\hline\\[-0.35cm]
$h$&0&63&65&68&100&102&103&108&109\\[-0.15cm]
$\widehat{F}_{53}(h)$&0/8&1/8&2/8&3/8&4/8&5/8&6/8&7/8&8/8\\[0.05cm]
$\widehat{G}_{53}(h)$&0&0.70&0.71&0.72&0.85&0.86&0.86&0.87&0.88\\[-0.05cm]
$|\widehat{G}_{53}(h)-\widehat{F}_{53}(h)|$&0&\textbf{0.57}&0.45&0.35&0.35&0.24&0.11&0.01&0.12\\\hline
$h$&0&2&5&37&39&40&45&46\\[-0.15cm]
$\widehat{F}_{116}(h)$&0/7&1/7&2/7&3/7&4/7&5/7&6/7&7/7\\[0.05cm]
$\widehat{G}_{116}(h)$&0&0.66&0.67&0.69&0.82&0.83&0.83&0.84\\[-0.05cm]
$|\widehat{G}_{116}(h)-\widehat{F}_{116}(h)|$&0&\textbf{0.52}&0.39&0.26&0.25&0.03&0.16\\\hline
$h$&0&3&35&37&38&43&44\\[-0.15cm]
$\widehat{F}_{118}(h)$&0/6&1/6&2/6&3/6&4/6&5/6&6/6\\[0.05cm]
$\widehat{G}_{118}(h)$&0&0.80&0.81&0.82&0.92&0.93&0.93\\[-0.05cm]
$|\widehat{G}_{118}(h)-\widehat{F}_{118}(h)|$&0&\textbf{0.63}&0.48&0.32&0.25&0.09&0.07\\\hline
$h$&0&32&34&35&40&41\\[-0.15cm]
$\widehat{F}_{121}(h)$&0/5&1/5&2/5&3/5&4/5&5/5\\[0.05cm]
$\widehat{G}_{121}(h)$&0&0.88&0.88&0.89&0.96&0.97\\[-0.05cm]
$|\widehat{G}_{121}(h)-\widehat{F}_{121}(h)|$&0&\textbf{0.68}&0.48&0.29&0.16&0.03\\\hline
\end{tabular}
\end{table}

Because we are unable to known when the next earthquake will occur, it is not possible to calculate an empirical CDF associated with any $F_t$. However, from Theorems \ref{p1} and \ref{p2}, the random CDF $\widehat{G}_t$ permits us to estimate this CDF. In addition, as illustrated next, Corollary \ref{c2} allows us to find approximate confidence bands for the unknown limit CDF $G$. Let

{
\begin{equation}\label{e10}
  [m^-_{t,\alpha}, m^+_{t,\alpha}] = \frac{1}{2}\left(\frac{x_\alpha^2}{t}+2\widehat{m}_t\pm\frac{x_\alpha}{\sqrt{t}}\sqrt{\frac{x_\alpha^2}{t}+4\widehat{m}_t}\right),
\quad 0<\alpha< 1,
\end{equation}
be an $(1-\alpha)\times 100\%$ confidence interval for $m$, with $t$ fixed. Hence, by choosing $x_\alpha$ such that $\pp(|\sqrt{{t}/m}(\widehat{m}_t-m)|\leq x_\alpha)=1-\alpha$, we have $[m^-_{t,\alpha}, m^+_{t,\alpha}]$ given in \eqref{e10} is an $(1-\alpha)\times 100\%$ confidence interval for $m$, with $t$ fixed. Consequently,
$$
\pp\Big(\bigcap_{h>0}\{G^-_{t,\alpha}(h)\leq G(h)\leq G^+_{t,\alpha}(h)\}\Big)=1-\alpha,
$$
where $G^-_{t,\alpha}(h)$ and $G^+_{t,\alpha}(h)$ are the CDFs corresponding to the exponential distribution with parameters $m^-_{t,\alpha}$ and $m^+_{t,\alpha}$, respectively, that is, $G^-_{t,\alpha}(h)=1-\exp(-m^-_{t,\alpha}h)$ and $G^+_{t,\alpha}(h)=1-\exp(- m^+_{t,\alpha} h)$, for $h\geq0$. Thus, $G^-_{t,\alpha}$ (lower band) and $G^+_{t,\alpha}$ (upper band) are $100\times (1-\alpha)\%$ confidence bands for the limit CDF $G$ over $h >0$, with $t$ fixed.
From Corollary \ref{c2}, for $\alpha=0.05$, we have $x_\alpha=1.96$. For this value of $\alpha$, $t=130$ and the data in Table \ref{tab:5}, we plot 95\% confidence bands for $G$ in Figure \ref{fig:2}, from which is possible to note that, within ten years more, one has a high probability of occurrence for a high intensity earthquake, whereas this probability is practically one within twenty years more.

\begin{figure}[h!]
\vspace{-1cm}
\centering
  \includegraphics[width=7cm,height=7cm]{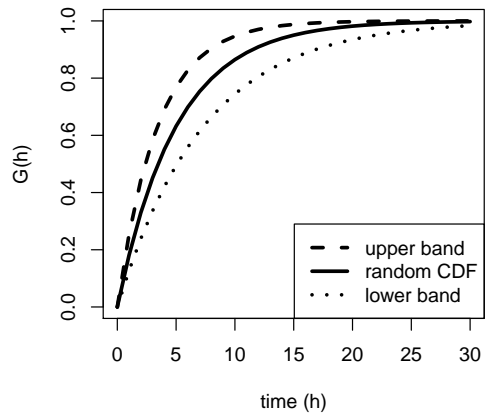}
\vspace{-0.65cm}
  \caption{95\% confidence bands for $G$ with $t=130$ for Chilean large earthquake data.}\label{fig:2}
\end{figure}

\section{Conclusions}\label{section8}

In this paper, we have proven that, by assuming large earthquakes occur after approximately $k-1$ seismic events of small intensity, the distribution of the occurrence time for the next large earthquake, by knowing the last seismic event occurred a long time ago, is exponential with rate depending on the asymptotic slope of the cumulative intensity function corresponding to a non-homogeneous Poisson process, which does not depend on $k$. We conclude that, for large values of $\tau$, it is advisable to estimate $m$, the asymptotic slope of the cumulative function, by $\widehat{m}_\tau=N_\tau/\tau$, where $N_\tau$ corresponds to the number of seismic events occurred in $[0, \tau]$. We have seen this estimator is consistent for $m$ and, in a number of cases, it turns out be a maximum likelihood estimator. Moreover, by means of $\widehat{m}_\tau$, a random cumulative distribution function is defined and it is proved that it satisfies results similar to the Glivenko-Cantelli and Kolmogorov theorems. Simulations carried out for $m=1$, $k=10$ and a p-value equals 0.057 suggested us that, for $\tau\geq25$ years, the random cumulative distribution function $\widehat{G}_\tau$ is a good approximation for $G$, the limit cumulative distribution function. Because it is not possible to known when the next earthquake will occur, an empirical distribution function for the waiting time of the next earthquake cannot be evaluated. However, what we have defined as the limit cumulative distribution function, along with suitable confidence bands for the unknown cumulative distribution function, has provided additional instruments to alert on an eventual large earthquake. Finally, a real-world data analysis has enabled us to illustrate the potential applications of our proposal.

\section*{Acknowledgements}

The authors gratefully acknowledge financial support from FONDECYT 1120879 grant of CONICYT-Chile.



\end{document}